\documentclass[11pt]{article}
\usepackage[latin1]{inputenc}
\usepackage{amsmath,amsthm,amssymb}
\usepackage{amsfonts}
\usepackage{amsmath,amsthm,amssymb,amscd}
\usepackage{latexsym}
\usepackage{color}
\usepackage{graphicx}
\usepackage{mathrsfs}
\usepackage{hyperref}

\makeatletter \@addtoreset{equation}{section} \makeatother

\newtheorem{theorem}{Theorem}[section]

\newtheorem{lemma}{Lemma}[section]
\newtheorem{remark}{Remark}[section]

\begin{document}

\title{\sc On $p$-Laplacian Kirchhoff-Schr\"{o}dinger-Poisson type systems with critical growth on the Heisenberg group}

\author{\sc
Shujie Bai$^{1}$, 
Yueqiang Song$^{1}$,  
Du\v{s}an D. Repov\v{s}$^{
2,
3,
4,
\thanks{
Corresponding author 
}
\ 
,
\thanks{
E-mail addresses:  
\href{mailto:annawhitebai@163.com}{annawhitebai@163.com},
\href{mailto:lsongyq16@mails.jlu.edu.cn}{songyq16@mails.jlu.edu.cn},
 \href{mailto:dusan.repovs@guest.arnes.si}{dusan.repovs@guest.arnes.si}
}
}$\\
$^{\small\mbox{1}}${\small College of Mathematics, Changchun Normal University,   Changchun, 130032,  P.R. China}\\
$^{\small\mbox{2}}${\small Faculty of Education, University of Ljubljana, Ljubljana, 1000, Slovenia }\\[-0.2cm]
$^{\small\mbox{3}}${\small Faculty of Mathematics and Physics, University of Ljubljana, Ljubljana, 1000, Slovenia }\\[-0.2cm]
$^{\small\mbox{4}}${\small Institute of Mathematics, Physics and Mechanics, Ljubljana, 1000, Slovenia}
}

\date{}
\maketitle

\begin{abstract}
In this article, we investigate the
Kirchhoff-Schr\"{o}dinger-Poisson type systems on the Heisenberg
group of the following form:
\begin{equation*}
\left\{
\begin{array}{lll}
{-(a+b\int_{\Omega}|\nabla_{H}
u|^{p}d\xi)\Delta_{H,p}u-\mu\phi |u|^{p-2}u}=\lambda
|u|^{q-2}u+|u|^{Q^{\ast}-2}u
&\mbox{in}\  \Omega, \\
-\Delta_{H}\phi=|u|^{p} &\mbox{in}\  \Omega, \\
u=\phi=0  &\mbox{on}\ \partial\Omega,
\end{array} \right.
\end{equation*}
where $a,b$ are positive real numbers, $\Omega\subset \mathbb{H}^N$
is a bounded region with smooth boundary, $1<p<Q$, $Q = 2N + 2$ is
the homogeneous dimension of the Heisenberg group $\mathbb{H}^N$,
$Q^{\ast}=\frac{pQ}{Q-p}$, $q\in(2p, Q^{\ast})$, and
$\Delta_{H,p}u=\mbox{div}(|\nabla_{H} u|^{p-2}\nabla_{H} u)$ is the
$p$-horizontal Laplacian. Under some appropriate conditions for the
parameters $\mu$ and $\lambda$, we establish existence and
multiplicity results for the system above. To some extent, we
generalize the results of An and Liu (Israel J. Math., 2020) and Liu
et al. (Adv. Nonlinear Anal., 2022).

\medskip

\medskip

\end{abstract}  

\maketitle

\section{Introduction and main results}
In this article, we investigate the following
Kirchhoff-Schr\"{o}dinger-Poisson type systems on the Heisenberg
group:
\begin{equation}\label{1.1}
\left\{
\begin{array}{lll}
{-(a+b\int_{\Omega}|\nabla_{H}
u|^{p}d\xi)\Delta_{H,p}u-\mu\phi |u|^{p-2}u}=\lambda
|u|^{q-2}u+|u|^{Q^{\ast}-2}u
&\mbox{in}\  \Omega, \\
-\Delta_{H}\phi=|u|^{p} &\mbox{in}\  \Omega, \\
u=\phi=0  &\mbox{on}\ \partial\Omega,
\end{array} \right.
\end{equation}
where $a,b$ are positive real numbers, $\Omega\subset \mathbb{H}^N$
is a bounded region with smooth boundary, $1<p<Q$, $Q = 2N + 2$ is
the homogeneous dimension of the Heisenberg group $\mathbb{H}^N$,
$Q^{\ast}=\frac{pQ}{Q-p}$, $q\in(2p, Q^{\ast})$,
$\Delta_{H,p}u=\mbox{div}(|\nabla_{H} u|^{p-2}\nabla_{H} u)$ is
known as the $p$-horizontal Laplacian, and $\mu$ and $\lambda$ are
some positive real parameters.

In recent years, geometrical analysis of the Heisenberg group has
found significant applications in quantum mechanics, partial
differential equations and other fields, which has attracted the
attention of many scholars who tried to establish the existence and
multiplicity of solutions of partial differential equations on the
Heisenberg group. For instance, in the subcase of problem \eqref{1.1}, when
$p=2$ and $b=\mu=0$, the existence of solutions for some nonlinear
elliptic problems in bounded domains has been established. Tyagi
\cite{ty} studied a class of singular boundary value problem on the
Heisenberg group:
\begin{equation}\label{1.2}
\left\{
\begin{array}{lll}-\Delta_{H} u  = \mu\frac{g(\xi)u}{(|z|^4+t^2)^{\frac{1}{2}}} + \lambda f(\xi, t),\quad
\xi\in  \Omega,\smallskip\smallskip\\
u|_{\partial\Omega} = 0,\end{array}\right.
\end{equation}
and under appropriate conditions,  obtained some existence
results by using Bonanno's three critical point theorem.
Goel et al. \cite{Go} dealt with  a class of Choquard type equation
on the Heisenberg group, and they established regularity of solutions and
nonexistence of solutions invoking the mountain pass theorem,
the linking theorem and iteration techniques and boot-strap method.

In the case $b\neq 0$ and $\mu=0$, problem \eqref{1.1} becomes the
Kirchhoff problem, which has also been widely studied. For example, Sun et
al. \cite{sun} dealt with the following Choquard-Kirchhoff problem
with critical growth:
$$M(\|u\|^2)(-\Delta_{H}u+u)= \int_{\mathbb{H}^N}\frac{|u(\eta)|^{Q_\lambda^\ast}}{|\eta^{-1}\xi|^\lambda}d\eta|u|^{Q_\lambda^\ast-2}u+\mu f(\xi,u),$$
where $f$ is a Carath\'{e}odory function, $M$ is the Kirchhoff
function, $\Delta_{H}$ is the Kohn Laplacian on the Heisenberg group
$\mathbb{H}^N$, $\mu>0$ is a parameter, and
$Q_\lambda^\ast=\frac{2Q-\lambda}{Q-2}$ is the critical exponent. In
their paper, a new version of the concentration-compactness
principle on the Heisenberg group was established for the first
time. Moreover, the existence of nontrivial solutions was obtained even
under nondegenerate and degenerate conditions. Zhou et al. \cite{z1}
proved the existence of solutions of Kirchhoff type nonlocal
integral-differential operators with homogeneous Dirichlet boundary
conditions on the Heisenberg group by using the variational method
and the mountain pass theorem.
 Deng and Xian \cite{de}
obtained the existence of solutions for Kirchhoff type systems
involving the
 $Q$-Laplacian operator on  the
 Heisenberg group with the help
of the Trudinger-Moser inequality and the mountain pass theorem. For
more related results, see \cite{Ben, bor, Liu2, pu4, pu5,
pu6, pu2, pu3}.

When $b=0$, $p=2$ and $\mu \neq 0$, problem \eqref{1.1}  becomes the
Schr\"{o}dinger-Poisson system. This is a very interesting subject
and has recently witnessed very profound results. For example,  An et
al. \cite{an} dealt with the following forms of  Schr\"{o}dinger-Poisson type
system
on the Heisenberg group:
\begin{equation*}
    \begin{cases}
    -\Delta_{H}u+\mu\phi u=\lambda|u|^{q-2}u+|u|^{2}u \quad &\mbox{in}\ \Omega,\\
    -\Delta_{H}\phi=u^2\quad &\mbox{in}\ \Omega,\\
    u=\phi=0\quad &\mbox{on}\ \partial \Omega,\\
    \end{cases}
\end{equation*}
where $\mu\in \mathbb{R}$ and $\lambda>0$ are some real
parameters and  $1<q<2$. By applying the concentration
compactness and the critical point theory,
they found at least two positive solutions and a
positive ground state solution.

Liang and Pucci \cite{liang1} studied the  following  critical
Kirchhoff-Poisson system on the Heisenberg group:
\begin{equation*}
\begin{cases}
 -M\left(\int_{\Omega}|\nabla_{H}u|^{2}d\xi\right)\Delta_{H}u+\phi|u|^{q-2}u= h(\xi, u)+ {\lambda}|u|^{2} u\quad  &\mbox{in} \ \Omega,\\
    -\Delta_{H}\phi=|u|^q\quad &\mbox{in}\ \Omega,\\
    u=\phi=0\quad &\mbox{on}\ \partial\Omega,
 \end{cases}
\end{equation*}
where $\Omega\subset \mathbb{H}^1$ is a smooth bounded
domain, $\Delta_H$ is the Kohn-Laplacian on the first Heisenberg group
$\mathbb{H}^1$ and $1 < q < 2$. By applying the symmetric mountain pass
lemma, they obtained the multiplicity of solutions with
$\lambda$ sufficiently small.

The Kirchhoff-Poisson system on the Heisenberg group with logarithmic and critical nonlinearity
was considered by Pucci and Ye \cite{pu3}:
\begin{equation*}
\begin{cases}
 -M\left(\int_{\Omega}|\nabla_{H}u|^{2}d\xi\right)\Delta_{H}u+\phi u=|u|^{2}u+ \lambda|u|^{q-2}u\ln|u|^2\quad  &\mbox{in} \ \Omega,\\
    -\Delta_{H}\phi=|u|^2\quad &\mbox{in}\ \Omega,\\
    u=\phi=0\quad &\mbox{on}\ \partial\Omega.
 \end{cases}
\end{equation*}
Under suitable assumptions on the Kirchhoff function $M$ covering the degenerate case,
they showed that for a sufficiently large $\lambda>0,$ there exists a nontrivial solution to the above problem.

When $p\neq2$ and $\mu \neq 0$, as far as we know, for Kirchhoff-
Schrodinger-Poisson systems \eqref{1.1} with critical nonlinearities
on the Heisenberg group, existence and multiplicity results are not
yet available. In the Euclidean case,  Du et al. \cite{d1} first
studied the existence results for the Kirchhoff-Poisson systems with
$p$-Laplacian under the subcritical case by using the mountain pass
theorem. Later, Du et al.  \cite{d2} studied quasilinear
Schr\"{o}dinger-Poisson systems. For the critical case, Du et al.
\cite{d3}  also obtained the existence of ground state solutions
with the variational approach.

Inspired by the above achievements, we aim to  establish some
results on the existence and multiplicity of nontrivial solutions of
the Kirchhoff-Schr\"{o}dinger-Poisson system \eqref{1.1}. The main
difficulties in dealing with problem \eqref{1.1} are
 the presence
of a nonlocal term and critical nonlinearities making the study of
this problem very challenging.

Before presenting the main results of this article, we first present some concepts of the Heisenberg group.
The Heisenberg group is represented by $\mathbb{H}^N$.
 If $\xi=(x,y,t)\in\mathbb{H}^N$, then the definition of this group operation is
$$\ \tau_{\xi}(\xi')=\xi\circ\xi'=(x+x',y+y',t+t'+2(x'y-y'x)),\    \mbox{for every}  \   \xi,\xi'\in \mathbb{H}^N.$$
$\xi^{-1}=-\xi$ is the inverse, and therefore
$(\xi')^{-1}\circ {\xi}^{-1}=(\xi\circ\xi')^{-1} $.

The definition of a natural group of dilations on $\mathbb{H}^N$ is
$\delta_s(\xi)=(sx,sy,s^2t)$, for every $s>0$. Hence,
$\delta_s(\xi_0\circ\xi) = \delta_s(\xi_0)\circ \delta_s(\xi)$. It
can be easily proved that the Jacobian determinant of dilatations
$\delta_s: \mathbb{H}^N\to \mathbb{H}^N$ is constant and equal to
$s^Q,$ for every $\xi=(x,y,t)\in \mathbb{H}^N$. The critical exponent
is $Q^{*}:= \frac{pQ}{Q-p}$, where the natural number $Q=2N+2$ is
called the homogeneous dimension of $\mathbb{H}^N$. We define the
Kor\'{a}nyi norm as follows
$$|\xi|_{H}=\left[(x^2+y^2)^2+t^2\right]^{\frac{1}{4}},  \  \mbox{for every}  \   \xi\in \mathbb{H}^N, $$
and we derive this norm from the Heisenberg group's anisotropic dilation.
Hence, the homogeneous degree of the Kor\'{a}nyi norm is equal to 1, in terms of dilations
$$\delta_s:(x,y,t)\mapsto(sx,sy,s^2t),  \  \mbox{for every}  \   s>0.$$
The set
$$B_{H}(\xi_{0},r)=\{\xi \in \mathbb{H}^{N}:d_{H}(\xi_0,\xi)<r\}, $$
denotes the  Kor\'{a}nyi open ball of radius $\mbox{r}$ centered at $\xi_{0}$.
For the sake of simplicity, we shall denote $B_{r} = B_{r}(O)$, where $O = (0,0)$ is
the natural origin of $\mathbb{H}^{N}$.

The following vector fields
$$T=\frac{\partial}{\partial t},  \
X_{j}=\frac{\partial}{\partial
x_{j}}+2y_{j}\frac{\partial}{\partial t},  \
Y_{j}=\frac{\partial}{\partial y_{j}}-2x_{j}\frac{\partial}{\partial
t},$$
generate the real Lie algebra of left invariant vector fields for $j=1,\cdots,n$,
which forms a basis satisfying
the Heisenberg regular commutation relation on $\mathbb{H}^{N}$.
This means that
$$[X_{j}, Y_{j}]=-4\delta_{jk}T,  \  [Y_{j},Y_{k}]=[X_{j},X_{k}]=[Y_{j},T]=[X_{j},T]=0.$$
The so-called horizontal vector field is just a vector field with the span of $[X_{j}, Y_{j}]^{n}_{j=1}$.

The Heisenberg gradient on $\mathbb{H}^{N}$ is  
$$\nabla_H =(X_{1},X_{2},\cdots,X_{n},Y_{1},Y_{2},\cdots,Y_{n}),$$
and the Kohn Laplacian on $\mathbb{H}^{N}$ is given by
$$\Delta_{H}=\sum_{j=1}^{N}X_{j}^{2}+Y_{j}^{2}=\sum_{j=1}^{N}[\frac{\partial^{2}}{\partial x_{j}^{2}}+\frac{\partial^{2}}{\partial y_{j}^{2}}+4y_{j}\frac{\partial^{2}}{\partial x_{j}\partial t}-4x_{j}\frac{\partial^{2}}{\partial x_{j}\partial t}+4(x_{j}^{2}+y_{j}^{2})\frac{\partial^{2}}{\partial t^{2}}].$$

The Haar measure is invariant under the left translations of the
Heisenberg group and is $Q$-homogeneous in terms of dilations. More
precisely, it is consistent with the $(2n + 1)$-dimensional Lebesgue measure.
Hence, as shown  Leonardi and Masnou \cite{Leonardi},  the topological dimension
$2N + 1$ of $\mathbb{H}^{N}$ is strictly less than its Hausdorff
dimension $Q = 2N + 2$. Next, $|\Omega|$ denotes the
(2N + 1)-dimensional Lebesgue measure of any measurable set $\Omega\subseteq
\mathbb{H}^{N}$. Hence,
$$|\delta_{s}(\Omega)|=s^{Q}|\Omega|,  \   d(\delta_{s}\xi)=s^{Q}d\xi,
\
\hbox{and}
\
|B_{H}(\xi_{0},r)|=\alpha_{Q}r^{Q},  \  \mbox{where}  \   \alpha_{Q}=|B_{H}(0,1)|.$$

Now, we can state  the main result of the paper.
\begin{theorem}\label{the1.1}
Let  $q\in(2p,  Q^{\ast})$. Then there exist  positive constants
$\mu_1$ and $\lambda_1$ such that for every $\mu\in(0,\mu_1)$ and $\lambda
\in (\lambda_1, +\infty)$, the following assertions hold:
\begin{itemize}
\item[$(I)$]  Problem \eqref{1.1} has a nontrivial weak solution;

\item[$(II)$] Problem \eqref{1.1} has
infinitely many nontrivial weak solutions if  parameter $a$ is
large enough.
\end{itemize}
\end{theorem}
We can give the following example for problem \eqref{1.1} with $p=3$
and $\Omega \subset \mathbb{H}^1$:
\begin{equation*}
\left\{
\begin{array}{lll}
{-(a+\int_{\Omega}|\nabla_{H} u|^{2}d\xi)\Delta_{H}u+\mu\phi
u}=\lambda |u|^{6}u+|u|^{10}u
&\mbox{in}\  \Omega, \\
-\Delta_{H}\phi=|u|^{3} &\mbox{in}\  \Omega, \\
u=\phi=0  &\mbox{on}\ \partial\Omega.
\end{array} \right.
\end{equation*}
In this case, $N=1$, $p=3$ and $q=8$, then $Q=2N+2=4$, $Q^\ast=12$.
If positive parameters $\mu$ small enough and $\lambda$ large
enough, by Theorem \ref{the1.1}, we know that problem \eqref{1.1}
has a nontrivial weak solution. Moreover, if in addition the
parameter $a$ is large enough,  problem \eqref{1.1} has infinitely
many nontrivial weak solutions. It should be noted that the methods
in An and Liu \cite{an}
and
Liang and Pucci
\cite{liang1} do not seem to apply to problem \eqref{1.1}.
\begin{remark}
Compared with previous results, this paper has the following key new
features:
\begin{itemize}
\item[$(1)$] The presence of the nonlocal  term $\phi|u|^{p-2}u$;

\item[$(2)$] The lack of compactness caused by critical index;

\item[$(3)$] The presence of the $p$-Laplacian makes this problem more complex and interesting.
\end{itemize}
It is worth stressing that the nonlocal term and the critical
exponent lead to the lack of compactness condition, and we use the
concentration-compactness principle  to overcome this difficulty.
Moreover, we shall  use some more refined estimates to overcome the
presence of the $p$-Laplacian.
\end{remark}

We need to emphasize here, that despite the similarity of some
properties between the classical Laplacian $\Delta$ and Kohn
Laplacian $\Delta_{H}$, similarities can be misleading (see Garofalo and Lanconelli
\cite{ga}), so there are still many properties that deserve further
study. Moreover, for the case $p\neq 2$, it is difficult to prove
the boundedness of Palais-Smale sequences. In order to overcome
these difficulties, we shall  use some more accurate estimates of
relevant expression. Additionally, we shall  use the concentration-compactness
principle on the Heisenberg group to prove the compactness
condition.

The paper is organized as follows. In Section \ref{s2}, we introduce
some notations and known facts. Moreover, we introduce some key
estimates. In addition, we define the corresponding energy
functional $I_{\lambda}$ and its derivative at
$u$, that is, $I_{\lambda}'(u)$. In Section \ref{s3}, we prove
Theorem \ref{the1.1}.

\section{Preliminaries and $(PS)_c$ condition}\label{s2}
\subsection{Preliminaries}

First of all, we collected some known facts, useful in
the sequel. For additional background material, readers are advised
to refer to
Papageorgiou et al. \cite{PPR}.

Let
$$\|u\|_s^{s}=\int_{\Omega}|u|^sd\xi, \    \hbox {for every}  \   u\in\,
L^{s}(\Omega),$$ represent the usual $L^s$-norm.

Following Folland and Stein \cite{fo}, we define the space
$\mathring{S}_1^2(\Omega)$ as the closure of
$C_{0}^{\infty}(\Omega)$ in $S_1^2(\mathbb{H}^{N})$. Then
$\mathring{S}_1^2(\Omega)$ is a Hilbert space with the
norm
$$\|u\|_{\mathring{S}_1^2(\Omega)}^2=\int_{\Omega}|\nabla_{H}u|^2d\xi.$$

We define the Folland-Stein space $S^{1,p}(\Omega)$ as the
closure of $C_0^{\infty}(\Omega)$ with the norm
$$\|u\|=(\int_{\Omega}|\nabla_{H}u|^pd\xi)^{\frac{1}{p}}.$$
Then the embedding
$$S^{1,p}(\Omega)\hookrightarrow L^s(\Omega), \ \mbox{for every}\ s \in (1, Q^{*}), $$ is compact. However, if
$s = Q^{*}$, the embedding is only continuous (see Vassiliev
\cite{DVassilev}).

Additionally, we say
that $(u,\phi)\in S^{1,p}(\Omega)\times S^{1,p}(\Omega)$ is a
solution of problem \eqref{1.1} if and only if
\begin{align*}
\begin{split}
&a\int_{\Omega}|\nabla_{H}u|^{p-2}\nabla_{H}u\nabla_{H}vd\xi+b\|u\|^{p}\int_{\Omega}|\nabla_{H}u|^{p-2}\nabla_{H}u\nabla_{H}vd\xi\\
&\ -\mu\int_{\Omega}\phi |u|^{p-2}uvd\xi-\lambda\int_{\Omega}|u|^{q-2}uvd\xi-\int_{\Omega}|u|^{Q^{*}-2}uvd\xi=0
\end{split}
\end{align*}
and
$$\int_{\Omega}\nabla_{H}\phi\nabla_{H}\omega d\xi-\int_{\Omega}|u|^{p}\omega d\xi=0,$$
for every $v,\omega \in S^{1,p}(\Omega)\times S^{1,p}(\Omega)$. Moreover, $(u,\phi)\in
S^{1,p}(\Omega)\times S^{1,p}(\Omega)$ is a positive solution of
problem \eqref{1.1} if $u$ and $\phi$ are both positive. Therefore,
in order
to apply the critical point theory,
we need to define the functional
$J(u,\phi): S^{1,p}(\Omega)\times S^{1,p}(\Omega)\rightarrow
\mathbb{R}$
as follows
$$J(u,\phi)=\frac{a}{p}\|u\|^{p}+\frac{b}{2p}\|u\|^{2p}+\frac{\mu}{2p}\int_{\Omega}|\nabla_{H}\phi|^2d\xi-\frac{\mu}{p}\int_{\Omega}\phi |u|^{p}d\xi-
\frac{\lambda}{q}\int_{\Omega}|u|^{q}d\xi-\frac{1}{Q^{*}}\int_{\Omega}|u|^{Q^{*}}d\xi,$$
for every $(u,\phi)\in S^{1,p}(\Omega)\times S^{1,p}(\Omega)$. Then
$J$ is $C^{1}$ on $S^{1,p}(\Omega)\times S^{1,p}(\Omega)$ and its
critical points are the solutions of problem \eqref{1.1}. Indeed,
the partial derivatives of $J$ at $(u,\phi)$ are denoted by
$J_u'(u,\phi)$, $J_\phi'(u,\phi)$, namely for every $v,\omega \in
S^{1,p}(\Omega)\times S^{1,p}(\Omega)$,
\begin{align*}
\begin{split}
J_u'(u,\phi)[v]=&a\int_{\Omega}|\nabla_{H}u|^{p-2}\nabla_{H}u\nabla_{H}vd\xi+b\|u\|^{p}\int_{\Omega}|\nabla_{H}u|^{p-2}\nabla_{H}u\nabla_{H}vd\xi\\
&\ -\mu\int_{\Omega}\phi |u|^{p-2}uvd\xi-\lambda\int_{\Omega}|u|^{q-2}uvd\xi-\int_{\Omega}|u|^{Q^{*}-2}uvd\xi=0
\end{split}
\end{align*}
and
$$J_\phi'(u,\phi)=\frac{\mu}{p}\int_{\Omega}\nabla_{H}\phi\nabla_{H}\omega d\xi-\frac{\mu}{p}\int_{\Omega}|u|^{p}\omega d\xi.$$
Standard computations show that $J_u'$ (respectively $J_\phi'$)
continuously maps $S^{1,p}(\Omega)\times S^{1,p}(\Omega)$ into   the
dual of $S^{1,p}(\Omega)$. Moreover, the functional $J$ is $C^{1}$
on $S^{1,p}(\Omega)\times S^{1,p}(\Omega)$ and
$$J_u'(u,\phi)=J_\phi'(u,\phi)=0$$ if and only if $(u,\phi)$ is a solution of problem \eqref{1.1}.
\begin{lemma}\label{lem2.1}
Let $u\in S^{1,p}(\Omega)$. Then there is a unique nonnegative function $\phi_u\in \mathring{S}_1^2(\Omega)$ such that
\begin{equation}\label{2.1}
\left\{
\begin{array}{lll}
-\Delta_{H}\phi=|u|^{p} &\mbox{in}\  \Omega, \\
\phi=0  &\mbox{on}\ \partial\Omega.
\end{array} \right.
\end{equation}
Furthermore, $\phi_u\geq0$ and $\phi_u>0$ if $u\neq0$. Also, \\
(i) $\phi_{tu}=t^{p}\phi_u,$ for every $t>0$;\\
(ii) $\|\phi_u\|_{\mathring{S}_1^2(\Omega)}\leq \hat{C}\|u\|^{p}$, where $\hat{C}>0$;\\
(iii) Let $u_n\rightharpoonup u$ in $S^{1,p}(\Omega)$. Then
$\phi_{u_{n}}\rightharpoonup\phi_u$ in $\mathring{S}_1^2(\Omega)$,
and
\begin{equation}\label{2.2}
\int_\Omega\phi_{u_n} |u_{n}|^{p-2}u_nvd\xi\rightarrow\int_\Omega\phi_{u} |u|^{p-2}uvd\xi,
\
\hbox{for every}
\
v\in S^{1,p}(\Omega).
\end{equation}
\end{lemma}

\begin{proof}
For any $u\in\mathring{S}_1^2(\Omega)$, we define $W: \mathring{S}_1^2(\Omega)\rightarrow\mathbb{R}$,
$$W(v)=\int_{\Omega}v|u|^{p}d\xi,\ \mbox{for every}\ v\in\mathring{S}_1^2(\Omega).$$
Let $v_n\rightarrow v\in\mathring{S}_1^2(\Omega),$ as $n\rightarrow\infty$.
It follows by the H\"{o}lder inequality that
\begin{align*}
\begin{split}
|W(v_n)-W(v)|&\leq\int_{\Omega}(v_n-v)|u|^{p}d\xi\\
&\leq(\int_{\Omega}|v_n-v|^{Q^{*}}d\xi)^{\frac{1}{Q^{*}}}(\int_{\Omega}|u|^{\frac{pQ^{*}}{Q^{*}-1}}d\xi)^{\frac{Q^{*}-1}{Q^{*}}}\\
&\leq
S^{-\frac{1}{p}}\|v_n-v\||u|_{\frac{pQ^{*}}{Q^{*}-1}}^{p}\rightarrow0,\
\mbox{as}\ n\rightarrow\infty,
\end{split}
\end{align*}
where
\begin{equation}\label{2.3}
S=\inf_{u\in
S^{1,p}(\Omega)\setminus\{0\}}\frac{\int_{\Omega}|\nabla_{H}
u|^{p}d\xi}{(\int_{\Omega}|u|^{Q^{*}}d\xi)^{\frac{p}{Q^{*}}}}
\end{equation}
is the best Sobolev constant. This implies that $W$ is a continuous linear functional. Using the
Lax-Milgram theorem, we see that there is a unique
$\phi_u\in\mathring{S}_1^2(\Omega)$ satisfying
\begin{equation}\label{2.4}
\int_{\Omega}\nabla_H\phi_u\nabla_Hvd\xi=\int_{\Omega}v|u|^{p}d\xi,\ \mbox{for every}\ v\in S^{1,p}(\Omega).
\end{equation}
Thus, $\phi_u\in\mathring{S}_1^2(\Omega)$ is the unique solution of problem \eqref{2.1}.
Moreover, applying the maximum principle, one has $\phi_u\geq0$ and $\phi_u>0$ if $u\neq0$.
Indeed, for every $t>0$, one has
$$-\Delta_H\phi_{tu}=t^{p}u^{p}=t^{p}(-\Delta_H\phi_{u})=-\Delta_H(t^{p}\phi_{u}).$$
Hence $\phi_{tu}=t^{p}\phi_{u}$ due to the uniqueness of $\phi_{u}$.

Furthermore, since $\phi_{u}\in S^{1,p}(\Omega)$, we can view it
as a text function in problem \eqref{2.1}.
Then  by \eqref{2.4}, the Sobolev inequality and the H\"{o}lder inequality, we have
(henceforth $C_0$, $C_1$, $C_2$ will denote positive constants)
\begin{align*}
\int_{\Omega}|\nabla_H\phi_u|^{2}d\xi&=\int_{\Omega}\phi_u|u|^{p}d\xi
\leq|\phi_u|_{L^{2}(\Omega)}|u|_{L^{2p}(\Omega)}^{p} \leq
C_1\|\phi_u\|_{\mathring{S}_1^2(\Omega)}\|u\|^{p}.
\end{align*}
Therefore, we get $\|\phi_u\|_{\mathring{S}_1^2(\Omega)}\leq C_1\|u\|^{p}$.

Since $u_n\rightharpoonup u$ in $S^{1,p}(\Omega)$, we can conclude that
$u_n\rightarrow u$ a.e. in $\Omega$ and $\{|u_n|^{p}\}$ is bounded
in $L^{2}(\Omega)$. Moreover, we have
$|u_n|^{p}\rightharpoonup|u|^{p}$ in $L^{2}(\Omega)$. Then  for every
$v\in\mathring{S}_1^2(\Omega)$, it follows that
$$\int_{\Omega}v|u_n|^{p}d\xi\rightarrow\int_{\Omega}v|u|^{p}d\xi,\ \mbox{as}\ n\rightarrow\infty.$$
Therefore, $\phi_{u_{n}}\rightharpoonup\phi_u$ in $\mathring{S}_1^2(\Omega)$.
By the H\"{o}lder inequality, the Sobolev inequality, and (ii), one has
\begin{align*}
\begin{split}
\int_\Omega|\phi_{u_n} |u_{n}|^{p-2}u_n|^{\frac{2p}{2p-1}}d\xi
&\leq|\phi_{u_n}|_{L^{2}(\Omega)}^{\frac{2p}{2p-1}}(\int_\Omega|u_n|^{2p}d\xi)^{\frac{p-1}{2p-1}}\\
&\leq C_0\|\phi_{u_n}\|_{\mathring{S}_1^2(\Omega)}^{\frac{2p}{2p-1}}(\int_\Omega|u_n|^{2p}d\xi)^{\frac{p-1}{2p-1}}\\
&\leq C_2\|\phi_{u_n}\|^{\frac{2p^{2}}{2p-1}}(\int_\Omega|u_n|^{2p}d\xi)^{\frac{p-1}{2p-1}}.
\end{split}
\end{align*}
Hence, $\{\phi_{u_n} |u_{n}|^{p-2}u_n\}$ is bounded in $L^{\frac{2p}{2p-1}}(\Omega)$. Since
$$\phi_{u_n} |u_{n}|^{p-2}u_n\rightarrow\phi_{u} |u|^{p-2}u, \ \mbox{a.e. in}\ \Omega,$$
 we get
\begin{equation*}
\int_\Omega\phi_{u_n} |u_{n}|^{p-2}u_nvd\xi\rightarrow\int_\Omega\phi_{u} |u|^{p-2}uvd\xi,
\
\hbox{for every}
\
v\in S^{1,p}(\Omega).
\end{equation*}
The proof of Lemma \ref{lem2.1} is complete.
\end{proof}

By similar arguments as in An and Liu \cite{an}, we can get the following result.
\begin{lemma}\label{lem2.2}
Let $\Psi(u)=\phi_u$ for every $u\in S^{1,p}(\Omega)$, where $\phi_u$ is as in Lemma \ref{lem2.1},
and let $$\Upsilon=\{(u,\phi)\in S^{1,p}(\Omega)\times S^{1,p}(\Omega):J_\phi'(u,\phi)=0\}.$$
Then $\Psi$ is $C^{1}$ and $\Upsilon$ is the graph of $\Psi$.
\end{lemma}

We define the corresponding energy functional $I_{\lambda}(u)=J(u,\phi_u)$ of problem \eqref{1.1} by
\begin{equation}\label{2.5}
I_{\lambda}(u)=\frac{a}{p}\|u\|^{p}+\frac{b}{2p}\|u\|^{2p}-\frac{\mu}{2p}\int_{\Omega}\phi_u |u|^{p}d\xi-
\frac{\lambda}{q}\int_{\Omega}|u|^{q}d\xi-\frac{1}{Q^{*}}\int_{\Omega}|u|^{Q^{*}}d\xi,
\
\hbox{for every}
\
u\in S^{1,p}(\Omega).
\end{equation}
Based on the definition of $J$ and Lemma \ref{lem2.2}, we can
conclude that $I_\lambda$ is of $C^{1}$.

\begin{lemma}\label{lem2.3}(see An and Liu \cite{an})
Let $(u,\phi)\in S^{1,p}(\Omega)\times S^{1,p}(\Omega)$. Then
$(u,\phi)$ is a critical point of $J$
if and only if $u$ is a
critical point of $I_{\lambda}$ and $\phi=\Psi(u)$, where $\Psi$
was defined in
Lemma \ref{lem2.2}.
\end{lemma}

According to Lemma \ref{lem2.3}, we know that a solution $(u,\phi_{u})$ of
problem \eqref{1.1} corresponds to a critical point $u$ of the
functional $I_\lambda$ with $\phi=\Psi(u)$ and
\begin{align}\label{2.6}
\begin{split}
\langle I_{\lambda}'(u), v\rangle =&a\int_{\Omega}|\nabla_{H}u|^{p-2}\nabla_{H}u\nabla_{H}vd\xi+b\|u\|^{p}\int_{\Omega}|\nabla_{H}u|^{p-2}\nabla_{H}u\nabla_{H}vd\xi\\
&\ -\mu\int_{\Omega}\phi_u |u|^{p-2}uvd\xi-\lambda\int_{\Omega}|u|^{q-2}uvd\xi-\int_{\Omega}|u|^{Q^{*}-2}uvd\xi,
\
\hbox{for every}
\
v\in S^{1,p}(\Omega).
\end{split}
\end{align}
Therefore, based on the above arguments, we shall  strive to use critical point theory
and some analytical techniques to prove the existence of critical points of functional $I_\lambda$.

\subsection{$(PS)_c$ condition}

In this subsection, our main focus will be on proving that the functional $I_{\lambda}$ satisfies the Palais-Smale condition.

\begin{lemma}\label{lem3.1}
Let $q\in(2p, Q^{\ast})$. Then there exists $\mu_1>0$ such that for
any $\mu<\mu_1$, the energy functional $I_\lambda$ satisfies
$(PS)_c$ condition, where
\begin{equation}\label{3.1}
c\in\left(0,\ (\frac{1}{q}-\frac{1}{Q^{*}})(aS)^{\frac{Q}{p}}\right)
\end{equation}
and $S$ is the best Sobolev constant given by \eqref{2.3}.
\end{lemma}
\begin{proof}
Let us assume that $\{u_{n}\}_{n}\subset S^{1,p}(\Omega)$ is a $(PS)_{c}$ sequence related to the
functional $I_{\lambda}$, that is,
\begin{equation}\label{3.2}
I_{\lambda}(u_{n})\rightarrow c\ \mbox{and}\ I_{\lambda}'(u_{n})\rightarrow0, \mbox{as}\ n\rightarrow\infty.
\end{equation}
It follows that
\begin{align}\label{3.3}
c+o(1)\|u_{n}\|&\nonumber=I_{\lambda}(u_{n})-\frac{1}{q}I_{\lambda}'(u_{n})u_{n}\\
&\nonumber\geq
a(\frac{1}{p}-\frac{1}{q})\|u_n\|^{p}+b(\frac{1}{2p}-\frac{1}{q})\|u_n\|^{2p}-\mu(\frac{1}{2p}-\frac{1}{q})\int_{\Omega}\phi_{u_n}|u_n|^{p}d\xi
+(\frac{1}{q}-\frac{1}{Q^{*}})\int_{\Omega}|u_n|^{Q^{*}}d\xi\\
&\geq a(\frac{1}{p}-\frac{1}{q})\|u_n\|^{p}+(b-\mu
\hat{C})(\frac{1}{2p}-\frac{1}{q})\|u_n\|^{2p},
\end{align}
where $\hat{C}$ is a positive constant given by Lemma
\ref{lem2.1}(ii).
Let $\mu_1 = \frac{b}{\hat{C}}$. By \eqref{3.3}, we
know that $(PS)_{c}$ sequence $\{u_{n}\}_{n}\subset S^{1,p}(\Omega)$
is bounded for every $\mu < \mu_1$. Thus, we may assume that
$u_n\rightharpoonup u$ weakly in $S^{1,p}(\Omega)$, and
$u_n\rightarrow u$ in $L^{s}(\Omega)$ with $1<s<Q^{*}$. Furthermore,
since $I_\lambda(u_n)=I_\lambda(|u_n|)$, we may also assume that
$u_n\geq0$ and $u\geq0$. Therefore, invoking the concentration compactness
principle on the Heisenberg group (see
Vassiliev
 \cite[Lemma 3.5]{DVassilev}),
we obtain
\begin{align}\label{3.4}
\begin{split}
&|\nabla_{H}u_n|^{p}d\xi\rightharpoonup d\omega\geq|\nabla_{H}u|^{p}d\xi+\Sigma_{j\in\Lambda}\omega_j\delta_{x_j},\\
&|u_n|^{Q^{*}}d\xi\rightharpoonup
d\nu=|u|^{Q^{*}}d\xi+\Sigma_{j\in\Lambda}\nu_j\delta_{x_j},
\end{split}
\end{align}
where $\{x_j\}_{j\in\Lambda}\subset\Omega$
is the most a
countable
set of distinct points, $\omega$ and $\nu$ in $\mathbb{H}^{N}$ are
two positive Radon measures, and $\{\omega_j\}_{j\in\Lambda}$,
$\{\nu_j\}_{j\in\Lambda}$ are nonnegative numbers. Moreover, we have
\begin{equation}\label{3.5}
\omega_j\geq S\nu_j^{\frac{p}{Q^{*}}}.
\end{equation}
Next, we
shall
 show that $\Lambda=\emptyset$.
Indeed,  assume that the hypothesis $\omega_j\neq0$ holds for some $j\in\Lambda$.
Then when
$\varepsilon>0$ is sufficiently small, we can find
$0\leq\psi_{\varepsilon,j}\leq1$ satisfying the following
\begin{equation}\label{3.6}
\begin{cases}
\psi_{\varepsilon,j}=1 &\mbox{in}   \   B_{H}(\xi_{j},\frac{\varepsilon}{2}),\\
\psi_{\varepsilon,j}=0 &\mbox{in}   \   \Omega\backslash B_{H}(\xi_{j},\varepsilon),\\
|\nabla_{H}\psi_{\varepsilon,j}|\leq\frac{2}{\varepsilon},
\end{cases}
\end{equation}
where $\psi_{\varepsilon,j}\in
C_{0}^{\infty}(B_{H}(\xi_{j},\varepsilon))$
is a cut-off function.
Clearly, $(u_n\psi_{\varepsilon,j})_n$ is bounded in $S^{1,p}(\Omega)$.
It follows from \eqref{3.2} and the boundedness of $(u_n\psi_{\varepsilon,j})_n$ that
$$\langle I'_\lambda(u_n),u_n\psi_{\varepsilon,j}\rangle\rightarrow 0, \ \mbox{as}\ n\rightarrow\infty,$$
that is,
\begin{align}\label{3.7}
&a(\int_{\Omega}|\nabla_{H}u_n|^{p}\psi_{\varepsilon,j}+\int_{\Omega}u_n|\nabla_{H}u_n|^{p-2}\nabla_{H}u_n\nabla_{H}\psi_{\varepsilon,j}d\xi)
+b\|u_n\|^{p}(\int_{\Omega}|\nabla_{H}u_n|^{p}\psi_{\varepsilon,j}\\
&\nonumber
+\int_{\Omega}u_n|\nabla_{H}u_n|^{p-2}\nabla_{H}u_n\nabla_{H}\psi_{\varepsilon,j}d\xi)
-\mu\int_{\Omega}\phi_{u_n}|u_{n}|^{p}\psi_{\varepsilon,j}d\xi
=\lambda\int_{\Omega}u_{n}^{q}\psi_{\varepsilon,j}d\xi+\int_{\Omega}u_{n}^{Q^{*}}\psi_{\varepsilon,j}d\xi+o(1).
\end{align}
It follows from the dominated convergence theorem that
\begin{equation*}
\int_{B_{H}(\xi_{j},\varepsilon)}|u_{n}|^{q}\psi_{\varepsilon,j}d\xi
\rightarrow\int_{B_{H}(\xi_{j},\varepsilon)}|u|^{q}\psi_{\varepsilon,j}d\xi, \,
\ \mbox{as}\ n\rightarrow\infty.
\end{equation*}
Hence, letting $\varepsilon\rightarrow0$, we get
\begin{equation}\label{3.8}
\lim_{\varepsilon\rightarrow0}\lim_{n\rightarrow\infty}\int_{B_{H}(\xi_{j},\varepsilon)}|u_{n}|^{q}\psi_{\varepsilon,j}d\xi=0.
\end{equation}
By Lemma \ref{lem2.1},
\begin{equation}\label{3.9}
\lim_{n\rightarrow\infty}\int_{\Omega}\phi_{u_{n}}|u_n|^{p-2}u_nud\xi=
\int_{\Omega}\phi_{u}|u|^{p}d\xi,
\end{equation}
and since $u_n\rightarrow u$ in $L^{s}(\Omega)$ with $1<s<Q^{*}$, one has
\begin{align}\label{3.10}
\begin{split}
\int_{\Omega}(\phi_{u_{n}}|u_n|^{p}-\phi_{u_{n}}|u_n|^{p-1}u)d\xi&\leq\int_{\Omega}|\phi_{u_{n}}||u_{n}|^{p-1}|u_n-u|d\xi\\
&\leq\left(\int_{\Omega}|\phi_{u_{n}}|u_{n}|^{p-1}|^{\frac{p}{p-1}}d\xi\right)^{\frac{p-1}{p}}|u_n-u|_{p}
\rightarrow0.
\end{split}
\end{align}
Combining
 \eqref{3.9} with \eqref{3.10}, we obtain that
\begin{equation*}
\lim_{n\rightarrow\infty}\int_{\Omega}\phi_{u_{n}}|u_n|^{p}d\xi=
\int_{\Omega}\phi_{u}|u|^{p}d\xi,
\end{equation*}
thus
\begin{equation}\label{3.11}
\lim_{\varepsilon\rightarrow0}\lim_{n\rightarrow\infty}\int_{B_{H}(\xi_{j},\varepsilon)}\phi_{u_n}|u_n|^{p}\psi_{\varepsilon,j}d\xi=
\lim_{\varepsilon\rightarrow0}\int_{B_{H}(\xi_{j},\varepsilon)}\phi_{u}|u|^{p}\psi_{\varepsilon,j}d\xi=0.
\end{equation}
Since
\begin{equation*}
\int_{B_{H}(\xi_{j},\varepsilon)}d\xi=\int_{B_{H}(0,\varepsilon)}d\xi=|B_H(0,1)|\varepsilon^{Q},
\end{equation*}
applying the H\"{o}lder inequality, we obtain
\begin{align}\label{3.12}
\lim_{\varepsilon\rightarrow0}\lim_{n\rightarrow\infty}\int_{\Omega}u_n|\nabla_{H}u_n|^{p-1}\nabla_{H}\psi_{\varepsilon,j}d\xi&\nonumber\leq
\lim_{\varepsilon\rightarrow0}\lim_{n\rightarrow\infty}(\int_{B_{H}(\xi_{j},\varepsilon)}|\nabla_{H}u_n|^{p}d\xi)^{\frac{p-1}{p}}
(\int_{B_{H}(\xi_{j},\varepsilon)}|u_n\nabla_{H}\psi_{\varepsilon,j}|^{p}d\xi)^{\frac{1}{p}}\\
&\nonumber\leq C\lim_{\varepsilon\rightarrow0}(\int_{B_{H}(\xi_{j},\varepsilon)}|u_n|^{p}|\nabla_{H}\psi_{\varepsilon,j}|^{p}d\xi)^{\frac{1}{p}}\\
&\nonumber\leq
C\lim_{\varepsilon\rightarrow0}(\int_{B_{H}(\xi_{j},\varepsilon)}|u_n|^{Q^{*}}d\xi)^{\frac{1}{Q^{*}}}
(\int_{B_{H}(\xi_{j},\varepsilon)}|\nabla_{H}\psi_{\varepsilon,j}|^{Q}d\xi)^{\frac{1}{Q}}=0.
\end{align}
By \eqref{3.4}, we have
\begin{equation}\label{3.13}
\lim_{\varepsilon\rightarrow0}\lim_{n\rightarrow\infty}\int_{\Omega}|\nabla_{H}u_n|^{p}\psi_{\varepsilon,j}d\xi\geq
\lim_{\varepsilon\rightarrow0}\left(\omega_j+\int_{B_{H}(\xi_{j},\varepsilon)}|\nabla_{H}u|^{p}\psi_{\varepsilon,j}d\xi\right)=\omega_j
\end{equation}
and
\begin{equation}\label{3.14}
\lim_{\varepsilon\rightarrow0}\lim_{n\rightarrow\infty}\int_{\Omega}u_{n}^{Q^{*}}\psi_{\varepsilon,j}d\xi=
\lim_{\varepsilon\rightarrow0}\left(\nu_j+\int_{B_{H}(\xi_{j},\varepsilon)}u_{n}^{Q^{*}}\psi_{\varepsilon,j}d\xi\right)=\nu_j.
\end{equation}
Therefore, by \eqref{3.7}-\eqref{3.14}, one gets $\nu_j\geq
a\omega_j$. It follows from \eqref{3.5} that $$\nu_j = 0
\quad\mbox{or}\quad \nu_j\geq (aS)^{\frac{Q}{p}}.$$ In fact, if
$\nu_j\geq (aS)^{\frac{Q}{p}}$ holds, therefore
 by \eqref{3.2} and
\eqref{3.4}, for every $\mu < \mu_1$, we have
\begin{align}\label{3.15}
\begin{split}
c&=\lim_{n\rightarrow\infty}\{I_{\lambda}(u_{n})-\frac{1}{q}I_{\lambda}'(u_{n})u_{n}\}\geq\lim_{n\rightarrow\infty}(\frac{1}{q}-\frac{1}{Q^{*}})\int_{\Omega}|u_n|^{Q^{*}}d\xi\\
&\geq (\frac{1}{q}-\frac{1}{Q^{*}})\nu_j \geq
(\frac{1}{q}-\frac{1}{Q^{*}})(aS)^{\frac{Q}{p}},
\end{split}
\end{align}
which contradicts \eqref{3.1}. Thus, $\Lambda=\emptyset$.
By \eqref{3.4} and $\Lambda=\emptyset$, we have
\begin{equation}\label{3.16}
\int_{\Omega}|u_n|^{Q^{*}}d\xi\rightarrow\int_{\Omega}|u|^{Q^{*}}d\xi.
\end{equation}

Let $$\lim_{n\rightarrow\infty}\|u_n\|^{p}=A.$$ If $A=0$, then
$u_n\rightarrow0$ in $S^{1,p}(\Omega)$. So assume now that $A>0$. By
\eqref{3.2}, we get
\begin{align}\label{3.17}
\begin{split}
&a\int_{\Omega}|\nabla_{H}u_n|^{p-2}\nabla_{H}u_n\nabla_{H}vd\xi+b A\int_{\Omega}|\nabla_{H}u_n|^{p-2}\nabla_{H}u_n\nabla_{H}vd\xi\\
&\ -\mu\int_{\Omega}\phi_{u_{n}}
|u_n|^{p-2}u_nvd\xi-\lambda\int_{\Omega}|u_n|^{q-2}u_nvd\xi-\int_{\Omega}|u_n|^{Q^{*}-2}u_n
vd\xi=o(1).
\end{split}
\end{align}
Let $v=u$ in \eqref{3.17}. Then
\begin{equation}\label{3.18}
a\|u\|^{p}+bA\|u\|^{p}-\mu\int_{\Omega}\phi_{u}|u|^{p}d\xi-
\lambda\int_{\Omega}|u|^{q}d\xi-\int_{\Omega}|u|^{Q^{*}}d\xi=0.
\end{equation}
By \eqref{3.2}, \eqref{3.4}, \eqref{3.16} and Lemma \ref{lem2.1}, one also has
\begin{equation}\label{3.19}
\lim_{n\rightarrow\infty}a\|u_n\|^{p}+bA\|u_n\|^{p}-\mu\int_{\Omega}\phi_{u}|u|^{p}d\xi-
\lambda\int_{\Omega}|u|^{q}d\xi-\int_{\Omega}|u|^{Q^{*}}d\xi=0.
\end{equation}
Thus, combining \eqref{3.18} and \eqref{3.19}, we get $\lim_{n\rightarrow\infty}\|u_n\|^{p}=\|u\|^{p}$.
Thus, we see that $u_n\rightarrow u$ in $S^{1,p}(\Omega)$ by the uniform convexity of $S^{1,p}(\Omega)$.
This completes the proof of Lemma \ref{lem3.1}.
\end{proof}

\section{Proof of Theorem \ref{the1.1}}\label{s3}

\subsection{Existence of a nontrivial weak solutions}

We need the following auxiliary
lemmas to prove our main result.
\begin{lemma}\label{lem3.2}
Let $q\in(2p, Q^{\ast})$ and $\mu\in(0,\mu_1)$. Then  functional
$I_{\lambda}$ satisfies the mountain
pass geometry, that is,\\
(i) There exist constants $\rho,\alpha>0$ satisfying $I_{\lambda}(u)|_{\partial B_{\rho}}\geq \alpha,$ for every $u\in S^{1,p}(\Omega)$;\\
(ii) There exists $e\in S^{1,p}(\Omega)\backslash\overline{B_{\rho}}$ satisfying $I_\lambda(e)<0$.
\end{lemma}

\begin{proof}
First, applying the H\"{o}lder inequality, we get
\begin{align}\label{3.21}
\begin{split}
I_{\lambda}(u)&=\frac{a}{p}\|u\|^{p}+\frac{b}{2p}\|u\|^{2p}-\frac{\mu}{2p}\int_{\Omega}\phi_{u}|u|^{p}d\xi-
\frac{\lambda}{q}\int_{\Omega}|u|^{q}d\xi-\frac{1}{Q^{*}}\int_{\Omega}|u|^{Q^{*}}d\xi\\
&\geq\frac{a}{p}\|u\|^{p}+\frac{b-\mu \hat{C}}{2p}\|u\|^{2p}-\frac{\lambda}{q}\int_{\Omega}|u|^{q}d\xi-\frac{1}{Q^{*}}\int_{\Omega}|u|^{Q^{*}}d\xi\\
&\geq\frac{a}{p}\|u\|^{p}-\frac{\lambda}{q}S^{-\frac{q}{p}}|\Omega|^{\frac{Q^{*}-q}{Q^{*}}}\|u\|^{q}-\frac{1}{Q^{*}}S^{-\frac{Q^{*}}{p}}\|u\|^{Q^{*}}\\
&=\|u\|^{p}\{\frac{a}{p}-\frac{\lambda}{q}S^{-\frac{q}{p}}|\Omega|^{\frac{Q^{*}-q}{Q^{*}}}\|u\|^{q-p}-\frac{1}{Q^{*}}S^{-\frac{Q^{*}}{p}}\|u\|^{Q^{*}-p}
\}.
\end{split}
\end{align}
Let
$$f(t)=\frac{a}{p}-\frac{\lambda}{q}S^{-\frac{q}{p}}|\Omega|^{\frac{Q^{*}-q}{Q^{*}}}t^{q-p}-\frac{1}{Q^{*}}S^{-\frac{Q^{*}}{p}}t^{Q^{*}-p}, \ \hbox{for every} \  t\geq0.$$
We now show that there exists a
constant $\rho>0$ satisfying
$f(\rho)\geq \frac{a}{p}$. We see that $f$ is a continuous function
on $[0,+\infty)$ and
$\lim_{t\rightarrow0^{+}}f(t)=\frac{a}{p}.$ Hence there exists
$\rho$ such that $f(t)\geq \frac{a}{p}-\varepsilon_1,$ for every $0\leq
t\leq \rho$, where $\rho$ is small enough such that $\|u\|=\rho$. If
we choose $\varepsilon_1=\frac{a}{2p}$, we have $f(t)\geq
\frac{a}{2p},$ for every $0\leq t\leq \rho$. In particular,
$f(\rho)\geq \frac{a}{2p}$ and we obtain $I_{\lambda}(u)\geq
\frac{a}{2p}\rho^{p}=\alpha$ for $\|u\|=\rho$. Hence assertion (i)
of Lemma \ref{lem3.2} holds.

Next, we shall show that assertion (ii) of Lemma \ref{lem3.2} also holds:
\begin{align}\label{3.22}
\begin{split}
I_{\lambda}(su)&=\frac{as^{p}}{p}\|u\|^{p}+\frac{bs^{2p}}{2p}\|u\|^{2p}-\frac{\mu s^{2p}}{2p}\int_{\Omega}\phi_{u}|u|^{p}d\xi-
\frac{\lambda s^{q}}{q}\int_{\Omega}|u|^{q}d\xi-\frac{s^{Q^{*}}}{Q^{*}}\int_{\Omega}|u|^{Q^{*}}d\xi\\
&\leq\frac{as^{p}}{p}\|u\|^{p}+\frac{bs^{2p}}{2p}\|u\|^{2p}-\frac{s^{Q^{*}}}{Q^{*}}\int_{\Omega}|u|^{Q^{*}}d\xi
\rightarrow -\infty  \   \rm{as}  \   s\rightarrow +\infty.
\end{split}
\end{align}
Thus, we can deduce that $I_{\lambda}(s_{0}u)<0$ and
$s_{0}\|u\|>\rho,$ for every $s_{0}$ large enough. Let $e=s_{0}u$.
Then $e$ is the desired function and the proof of (ii) of Lemma
\ref{lem3.2} is complete.
\end{proof}

\begin{proof}[Proof of Theorem \ref{the1.1}(I)]

We claim that
\begin{equation}\label{3.23}
0<c_{\lambda}=\inf_{h\in\Gamma}\max_{0\leq
s\leq1}I_{\lambda}(h(s))<(\frac{1}{q}-\frac{1}{Q^{*}})(aS)^{\frac{Q}{p}},
\end{equation}
where $$\Gamma=\{h\in C([0,1],S^{1,p}(\Omega)):h(0)=1,h(1)=e\}.$$
Indeed, we can choose $v_{1}\in S^{1,p}(\Omega)\backslash\{0\}$ with
$\|v_{1}\|=1$. From \eqref{3.22}, we have
$\lim_{s\rightarrow+\infty}I_{\lambda}(sv_{1})=-\infty.$ Then
$$\sup_{s\geq0}I_{\lambda}(sv_{1})=I_{\lambda}(s_{\lambda}v_{1}),
\ \ \hbox{for some} \ s_{\lambda}>0.$$ So $s_{\lambda}$ satisfies
\begin{equation}\label{3.241}
as_{\lambda}^{p}\|v_{1}\|^{p}+bs_{\lambda}^{2p}\|v_{1}\|^{2p}=\mu s_{\lambda}^{2p}\int_{\Omega}\phi_{v_{1}}|v_{1}|^{p}d\xi+
\lambda s_{\lambda}^{q}\int_{\Omega}|v_{1}|^{q}d\xi+s_{\lambda}^{Q^{*}}\int_{\Omega}|v_{1}|^{Q^{*}}d\xi.
\end{equation}
Next, we shall prove that $\{s_{\lambda}\}_{\mu>0}$ is bounded. In fact,
suppose that the following hypothesis $s_{\lambda}\geq1$ is satisfied
for every $\lambda>0$. Then  it follows from \eqref{3.241} that
\begin{align}\label{3.25}
\begin{split}
(a+b)s_{\lambda}^{2p}&\geq
as_{\lambda}^{p}\|v_{1}\|^{p}+bs_{\lambda}^{2p}\|v_{1}\|^{2p}=\mu
s_{\lambda}^{2p}\int_{\Omega}\phi_{v_{1}}|v_{1}|^{p}d\xi+
\lambda s_{\lambda}^{q}\int_{\Omega}|v_{1}|^{q}d\xi+s_{\lambda}^{Q^{*}}\int_{\Omega}|v_{1}|^{Q^{*}}d\xi\\
&\geq
s_{\lambda}^{Q^{*}}\int_{\Omega}|v_{1}|^{Q^{*}}d\xi.
\end{split}
\end{align}
Since $2p<q<Q^{\ast}$, we can deduce that
$\{s_{\lambda}\}_{\lambda>0}$ is bounded.

Next, we shall demonstrate that $s_{\lambda}\rightarrow0,$ as $\lambda\rightarrow\infty$. Suppose to the contrary, that there exist $s_{\lambda}>0$ and a sequence $(\lambda_{n})_{n}$ with $\lambda_{n}\rightarrow\infty,$ as $n\rightarrow\infty,$ satisfying $s_{\lambda_{n}}\rightarrow s_{\lambda},$ as $n\rightarrow\infty$.
Invoking the Lebesgue dominated convergence theorem, we see that
$$\int_{\Omega}|s_{\lambda_{n}}v_{1}|^{q}d\xi\rightarrow\int_{\Omega}|s_{\lambda}v_{1}|^{q}d\xi,\      \mbox{as}   \   n\rightarrow\infty.$$
It now follows that
$$\lambda_{n}\int_{\Omega}|s_{\lambda}v_{1}|^{q}d\xi\rightarrow\infty, \    \mbox{as}   \   n\rightarrow\infty.$$
Thus, invoking \eqref{3.241}, we can show that this cannot happen.
Therefore, $s_{\lambda}\rightarrow0,$ as $\lambda\rightarrow\infty$.

Furthermore,  \eqref{3.241}  implies that
$$\lim_{\lambda\rightarrow\infty}\lambda\int_{\Omega}|s_{\lambda}v_{1}|^{q}d\xi=0$$
and
$$\lim_{\lambda\rightarrow\infty}\int_{\Omega}|s_{\lambda}v_{1}|^{Q^{*}}d\xi=0.$$
Hence based on the definition of $I_{\lambda}$ and $s_{\lambda}\rightarrow0,$ as $\lambda\rightarrow\infty$, we get that
$$\lim_{\lambda\rightarrow\infty}(\sup_{s\geq0}I_{\lambda}(sv_{1}))=\lim_{\lambda\rightarrow\infty}I_{\lambda}(s_{\lambda}v_{1})=0.$$
So, there is
$\lambda_{1}>0,$ satisfying for every $\lambda>\lambda_{1}$,
$$\sup_{s\geq0}I_{\lambda}(sv_{1})<(\frac{1}{q}-\frac{1}{Q^{*}})(aS)^{\frac{Q}{p}}.$$
Letting $e=t_{1}v_{1}$ with $t_{1}$ large enough for
$I_{\lambda}(e)<0$, we get
$$0<c_{\lambda}\leq\max_{0\leq s\leq1}I_{\lambda}(h(s)),  \   \hbox{where}   \   h(s)=st_{1}v_{1}.$$
Therefore
$$0<c_{\lambda}\leq\sup_{s\geq0}I_{\lambda}(sv_{1})<(\frac{1}{q}-\frac{1}{Q^{*}})(aS)^{\frac{Q}{p}},$$
for $\lambda$ large enough.   This completes the proof of Theorem
\ref{the1.1}(I).
\end{proof}

\subsection{Existence of infinitely many nontrivial weak solutions }

In this subsection, we shall use the
 Krasnoselskii genus theory to prove Theorem \ref{the1.1}(II).  To this end,
let  $E$ be a Banach space and  denote by $\Lambda$ the class of all closed subsets
$A \subset E\backslash \{0\}$ that are
symmetric with respect to the origin, that is,
$ u\in E$ implies $-u\in E$. Moreover, suppose that $X$ is $k$-dimensional and
$X=\hbox{span}\{z_{1},\cdots,z_{k}\}$. For every $n\geq k$, inductively select
$z_{n+1}\notin X_{n}=\hbox{span}\{z_{1},\cdots,z_{n}\}$. Let $R_n = R(Z_n)$
and $\Upsilon_n=B_{R_{n}}\bigcap Z_n$. Define
$$W_n=\{\varphi\in C(\Upsilon_n,E):\varphi|\partial_{B_{R_{n}}\bigcap Z_n}=id \ \mbox{and}\ \varphi \ \mbox{is odd}\}$$
and
$$\Gamma_{i}=\{\varphi(\overline{\Upsilon_n\setminus V}):\varphi\in W_n,n\geq i,V\in\Lambda,\Lambda\ \mbox{is closed},\gamma(V)\leq n-i\},$$
where $\gamma(V)$ is the Krasnoselskii genus of $V$.
\begin{theorem}\label{the4.1}(see   Rabinowitz \cite[Theorem 9.12]{Rabinowitz})
Let $I\in C^{1}(E, \mathbb{R})$ be even with $I(0) = 0$ and let E be an infinite-dimensional Banach space. Assume that $X$
is a finite-dimensional space, $E=X\oplus Y$, and that I satisfies the following properties:\\
(i) There exists $\theta>0$ such that $I$ satisfies $(PS)_{c}$ condition, for every $c\in(0,\theta)$;\\
(ii) There exist $\rho,\alpha>0$ satifying $I(u)\geq \alpha,$ for every $u\in\partial B_{\rho}\bigcap Y$;\\
(iii) For every finite-dimensional subspace $\tilde{E}\subset E$, there exists $R=R(\tilde{E})>\rho$ such that $I(u)\leq0$ on $\tilde{E}\setminus B_{R}$.\\
\noindent For every $i\in N$, let
$c_i=\inf_{X\in\Gamma_{i}}\max_{u\in Z}I(u),$
hence, $0\leq c_{i}\leq c_{i+1}$ and $c_{i}<\theta,$ for every $i>k$. Then every $c_{i}$ is a critical value of $I$. Moreover, if
$c_{i}=c_{i+1}=\cdots=c_{i+p}=c<\theta$ for $i>k$, then
$\gamma(K_{c})\geq p+1$, where
$$K_{c}=\{u\in E:I(u)=c \ \mbox{and}\ I'(u)=0\}.$$
\end{theorem}
\begin{lemma}\label{lem4.1}
There is a nondecreasing sequence $\{s_n\}$ of positive real numbers, independent of $\lambda,$ such that for every $\lambda>0$,
we have
$$c_n^{\lambda}=\inf_{W\in\Gamma_{n}}\max_{u\in W}I_\lambda(u)<s_n,$$
where $\Gamma_{n}$ was defined in Theorem \ref{the4.1}.
\end{lemma}

\begin{proof}
By the definition of $\Gamma_{n}$, one has
$$c_n^{\lambda}\leq\inf_{W\in\Gamma_{n}}\max_{u\in W}\{\frac{a}{p}\|u_n\|^{p}+\frac{b}{2p}\|u_n\|^{2p}-\frac{\mu}{2p}\int_{\Omega}\phi_{u_{n}}u_{n}^{p}d\xi-\frac{1}{Q^{*}}\int_{\Omega}|u_{n}|^{Q^{*}}d\xi\}=s_n,$$
therefore $s_n<\infty$ and $s_n\leq s_{n+1}$.
\end{proof}

\begin{proof}[Proof of Theorem \ref{the1.1}(II)] We note that
$I_{\lambda}$ satisfies $I_{\lambda}(0)=0$ and
$I_{\lambda}(-u)=I_{\lambda}(u)$. In the sequel, we shall divide the
proof into the following three steps:

\noindent{\bf Step 1.} We shall  prove that $I_\lambda$ satisfies hypothesis
(ii) of Theorem \ref{the4.1}.
 Indeed, similar to the proof of (i)
in Lemma \ref{lem3.2}, we can easily prove that the energy
functional $I_{\lambda}$ satisfies the hypothesis  $(ii)$ of Theorem
\ref{the4.1}.

\noindent{\bf Step 2.} We shall  prove that $I_\lambda$ satisfies hypothesis
(iii) of Theorem \ref{the4.1}. Indeed, let $Y$ be a
finite-dimensional subspace of $S^{1,p}(\Omega)$. Since all norms in
finite-dimensional space are equivalent, it follows that for every $u
\in Y$,  we have
\begin{align}\label{3.66}
I_{\lambda}(u)\nonumber&\leq\frac{a}{p}\|u\|^{p}+\frac{b}{2p}\|u\|^{2p}-\frac{1}{Q^{*}}\int_{\Omega}|u|^{Q^{*}}d\xi\\
&\leq
\frac{a}{p}\|u\|^{p}+\frac{b}{2p}\|u\|^{2p}-\frac{1}{Q^{*}}C\|u\|^{Q^\ast},
\end{align}
for some positive constant $C > 0$. Also, because of $2p < Q^{*}$,
we can choose a large $R > 0$ such that $I_{\lambda}(u) \leq 0$ on
$S^{1,p}(\Omega)\setminus B_R$. This fact implies that the energy
functional $I_{\lambda}$ satisfies the hypothesis  $(iii)$ of
Theorem \ref{the4.1}.

\noindent{\bf Step 3.} We shall  prove that problem \eqref{1.1} has
infinitely many nontrivial weak solutions. Indeed, applying the
argument in Wei and Wu \cite{Wei}, we can choose $a_1$ large enough
so that for every $a>a_1$,
$$\sup s_n<(\frac{1}{q}-\frac{1}{Q^{*}})(aS)^{\frac{Q}{p}},$$
that is,
$$c_n^{\lambda}<s_n<(\frac{1}{q}-\frac{1}{Q^{*}})(aS)^{\frac{Q}{p}}.$$
 Thus, one has
$$0<c_{1}^{\lambda}\leq c_{2}^{\lambda}\leq\cdots\leq c_{n}^{\lambda}<s_n<(\frac{1}{q}-\frac{1}{Q^{*}})(aS)^{\frac{Q}{p}}.$$
From Lemma \ref{lem3.1}, we know that $I_{\lambda}$  satisfies
$(PS)_{c_i^\lambda} (i=1,2,\cdots,n)$ condition. This fact implies
that the levels $c_{1}^{\lambda}\leq c_{2}^{\lambda}\leq\cdots\leq
c_{n}^{\lambda}$ are critical values of $I_{\lambda}$, which be
guaranteed
 by an application of the Rabinowitz result
\cite[Proposition 9.30]{Rabinowitz}.

If $c_{i}^{\lambda}=c_{i+1}^{\lambda}$ where $i = 1,2,\cdots,k-1$,
then applying the Ambrosetti and Rabinowitz result \cite[Remark 2.12 and
Theorem 4.2]{am}, we see that the set $K_{c_{i}^{\lambda}}$ consists
of infinite number of different points, so problem \eqref{1.1} has
infinite number of weak solutions. Hence, problem \eqref{1.1} has at
least $k$ pairs of solutions. Since $k$ is arbitrary, we can
conclude that problem \eqref{1.1} has infinitely many solutions.
This completes the proof of Theorem \ref{the1.1}(II).
\end{proof}



\subsection*{Acknowledgments}
Song was supported by the National Natural Science Foundation of
China (No. 12001061), the Science and Technology Development Plan
Project of Jilin Province, China (No. 20230101287JC) and Innovation
and Entrepreneurship Talent Funding Project of Jilin Province
(No. 2023QN21). Repov\v{s} was supported by the Slovenian
Research Agency program No. P1-0292 and grants Nos. J1-4031,
J1-4001, N1-0278, N1-0114, and N1-0083.

\subsection*{Authors' ORCID ID's}
\noindent
Shujie Bai: https://orcid.org/0000-0002-0821-2388 \\
Yueqiang Song: https://orcid.org/0000-0003-3570-3956\\
Du\v{s}an D. Repov\v{s}: https://orcid.org/0000-0002-6643-1271

\end{document}